\documentclass[11pt,a4paper]{article}
\usepackage[a4paper]{geometry}
\usepackage[utf8]{inputenc} 
\usepackage[english]{babel}
\usepackage{epsfig}
\usepackage{amsmath, amsfonts, amstext, amscd, bezier,dsfont,bm}
\usepackage{amssymb,amsthm}
\usepackage{indentfirst}
\usepackage{graphicx}
\usepackage{float}
\usepackage{color}
\usepackage{hyperref}
\usepackage[affil-it]{authblk}
\usepackage{ulem}

\newcommand*{\Scale}[2][4]{\scalebox{#1}{$#2$}}

\usepackage{amssymb} 
\usepackage[authoryear]{natbib}

\usepackage{mathtools}
\usepackage{mathrsfs} 
\usepackage{seqsplit}

\usepackage{enumitem}
\newcommand{\subscript}[2]{$#1 _ #2$}
\newcommand{\norm}[1]{\left\lVert#1\right\rVert}
\newcommand{\Zd}[1]{\mathbb{Z}^{#1}}
\newcommand{\Rd}[1]{\mathbb{R}^{#1}}
\newcommand{\g}[1]{\mathbf{#1}}
\newcommand{\A}{\bold{A}}
\newcommand{\B}{\bold{B}}
\renewcommand{\b}{\bold{b}}

\newcommand{\X}{\mathcal{X}}

\renewcommand{\P}{\mathbb{P}}

\newcommand{\til}[1]{\widetilde{#1}}
\newcommand{\reci}[1]{\mathbf{R}(\g{#1})}
\newcommand{\rec}{\mathbf{R}}
\newcommand{\kept}{\mathbf{K}}
\newcommand{\rect}[2]{\mathbf{R}[#1,#2]}
\newcommand{\basis}[1]{\text{Basis}(#1)}
\newcommand{\death}[1]{\text{Death}(#1)}
\newcommand{\birth}[1]{\text{Birth}(#1)}
\newcommand{\life}[1]{\text{life}(#1)}
\newcommand{\flag}[1]{\text{Flag}(#1)}
\newcommand{\suppv}[1]{\text{Supp}_v(#1)}
\newcommand{\suppe}[1]{\text{Supp}_e(#1)}

\newcommand{\Y}{\mathcal{Y}}
\newcommand{\ind}{\mathds{1}}

\newcommand{\esp}{\mathbb{E}}

\newcommand{\dg}[2]{d^{\, #1}_{#2}}
\newcommand{\free}[1]{\xi^{\,#1}_{\,t}}
\newcommand{\dep}[2]{\eta^{\,#1 ,\, #2}_{\,t}}
\newcommand{\Q}{Q}

\theoremstyle{plain}
\newtheorem{theorem}{Theorem}
\newtheorem{lemma}[theorem]{Lemma}

\newtheorem{proposition}[theorem]{Proposition}
\newtheorem{example}[subsection]{Example}
\newtheorem*{example*}{Example}
\theoremstyle{definition}

\theoremstyle{remark}

\numberwithin{equation}{section}
\numberwithin{theorem}{section}

\usepackage{mathtools}
\mathtoolsset{showonlyrefs} 

\providecommand{\keywords}[1]{\textit{Keywords:} #1}
\makeatletter
\let\@fnsymbol\@arabic
\makeatother

\begin{document}

\title{\Large{Graphical Construction of Spatial Gibbs Random Graphs}}

\author{\normalsize{Andressa Cerqueira\footnote{Universidade Federal de S\~ao Carlos, Brazil.}\,\, and Nancy L. Garcia\footnote{Universidade Estadual de Campinas, Brazil.}}}

\date{\today}
\maketitle

\begin{abstract}
We consider a Random Graph Model on $\Zd{d}$ that incorporates the interplay between the statistics of the graph and the underlying space where the vertices are located. Based on a graphical construction of the model as the invariant measure of a birth and death process, we prove the existence and uniqueness of a measure defined on graphs with vertices in $\Zd{d}$ which coincides with the limit along the measures over graphs with finite vertex set. As a consequence, theoretical properties such as  exponential mixing of the infinite volume measure and central limit theorem for averages of a real-valued function of the graph are obtained. Moreover,  a perfect simulation algorithm based on the clan of ancestors is described in order to sample a finite window of the equilibrium measure defined on $\Zd{d}$. \end{abstract}


\section{Introduction}

In recent years there has been an increasing interest in the study of probabilistic
models defined on graphs in order to describe the random interactions in a complex  network. The Erd\"os-R\'enyi random graph model \citep{erdHos1960evolution} might be the most famous random graph model. Despite of its well known properties, this model presents independent edges making it unsuitable to model networks that exhibt dependencies between edges.  The Exponential Random Graph Model (ERGM), that was pioneered by \cite{frank1986markov}, allows the representation of a large number of dependencies found in real networks, such as social networks \citep{robins2007introduction,robins2007recent}.  ERGM is a family of probability distributions on graphs belonging to the exponential family  (known as Gibbs distributions in Statistical Physics) such that the probability of a given graph depends only on  the statistics of the graph, such as number of edges, stars, triangles and so on. 
For an overview of some specifications to represent structural relations in ERGM see \cite{snijders2006new} and references therein. Different problems related with the ERGM have been studied in the literature, for example parameter inference \citep{caimo2011bayesian,hummel2012improving}, consistency results \citep{shalizi2013consistency}, mixing time properties \citep{bhamidi2008mixing}, graph limits \citep{chatterjee2013estimating}, Markov Chain Monte Carlo methods \citep{snijders2002markov} and perfect simulation algorithms \citep{butts2015novel,cerqueira2020note}.

Despite its applicability, this model does not incorporate properties of the underlying space where the vertices are located. For some real networks, it is reasonable to consider that the network connections might depend on the physical space where the graph is embedded. The Random Geometric Graph is the most studied spatial network model. In this model, proposed in the seminal work \cite{gilbert1961random}, the nodes of the graph are chosen randomly in a metric space and two nodes are connected by an edge when their distance does not exceed some fixed value. For more details about this model see \cite{penrose2003random}.

In order to combine the information about the spatial distance between the vertices of the graph and the graph statistics, the works \cite{ferrari2010gibbs}  and \cite{mourrat2018spatial} proposed different spatial random graph models governed by Gibbs-like measures. In \cite{ferrari2010gibbs}, the authors define a Gibbs measure on graphs with vertices located in $\mathbb{Z}^d$ that penalizes long edges and high degree.  On the other hand, \cite{mourrat2018spatial} introduce and study the behavior of a Spatial Gibbs Random Graph defined on an one-dimensional space that gives more weight to graphs with small average distance between vertices. They also consider that the existence of each edge in the graph has a cost that depends only on the underlying space. \cite{endo2020spatial} studied the local convergence properties of this model through its local limits. 

In this paper we introduce a random graph model that describes a balance between the statistics of the graph and the distance between the vertices in the underlying space. For a finite vertex set $V\subset \Zd{d}$, we define a Gibbs measure $\mu_{_V}$ with weights depending  on the length of the edges and a sufficient statistic which is a function of the graph that describes the interaction among the edges. Two natural questions arise: (1) existence: is there an infinite volume measure $\mu$ as the limit along the finite volume measures $\mu_{_V}$? And (2) uniqueness: if yes, is it unique? A construtive way to answer positively these questions is provided by a graphical construction of the Gibbs infinite measure as the invariant measure of a Markov process which can be seen as a birth-and-death process of edges. The graphical  construction is based on the {\it clan of ancestors} \citep{fernandez2001loss} of the spatial Gibbs measure and we prove the existence and uniqueness of an infinite volume measure $\mu$ which coincides with the limit along the finite volume measures $\mu_{_V}$ under some sufficient conditions.  To our knowledge, \cite{ferrari2010gibbs} is the first attempt to study the existence of such limits for a  Spatial Gibbs Random Graph measure that favors graphs with short edges and penalizes vertices with degree other than one.  Their results also involve percolation properties of the Gibbs measure. Their model is a particular case of the model proposed in this work. Since the uniqueness of the infinite measure was not a problem addressed by the authors, our results can be seen as an complement to their seminal work. 

%
Several properties are obtained from the clan of ancestors representation. For example, in the infinite volume, a vertex of the graph can be connected with infinitely many other vertices allowing a vertex to have infinite degree, this is not the case for the proposed model where we prove that $\mu$ is concentrated on the set of graphs with finite degree. Other results are exponential mixing of the infinite volume measure and central limit theorem for averages of a real-valued function of the graph. Moreover,  we obtain a perfect simulation scheme that delivers directly samples of finite windows of the infinite-volume measure without the need to perform limits such as $V \rightarrow \Zd{d}$. 

The {\it clan of ancestors} graphical construction used in this paper was originally proposed in \cite{fernandez2001loss} for measures that are absolutely continuous with respect to unit Poisson process on $\Rd{d}$.  Their construction, as well as ours, are based on the graphical representation of a Markovian process defined on the set of graphs that has $\mu$ as invariant measure. In the random graph model proposed here,  edges try to appear in the graph with an exponential rate, but some births are aborted according to some probability depending on the present configuration of the graph. Edges are removed from the graph at rate $1$. The process described above is dominated by a birth and death process for which edges are added in the graph every time they try to be born. This dominating process allows the presence of multiples edges in the graph giving rise to an independent multigraph process. To use the independent multigraph process to determine whether an edge $\{i,j\}$ is present at a time $t$ of the dependent process, it is necessary to look back in the past to the edges born before $\{i,j\}$ that could have an influence on the existence of $\{i,j\}$ at time $t$.  Once the clan is determined, it is necessary to perform a cleaning procedure forward on time to erase the edges that should not have been added in the graph. This construction directly induces a perfect simulation algorithm in order to sample a subgraph from $\mu$, as proposed in \cite{ferrari2002perfect}. This algorithm does not assume any monotonicity property of the process used in its construction as it is required in the case of perfectly sampling methods for the  ERGM \citep{cerqueira2020note}. Although the theorems follow straightforward from the graphical construction and the arguments of \cite{fernandez2001loss}  specialised to our setting, the results obtained for random graphs are novel and complement the ones found in the literature.

This paper is organized as follows. In Section \ref{sec:def} we introduce the definitions and notation to be used along the paper. Section \ref{sec:main} contains the main results and some examples. Sections \ref{sec:graphical_representation} and \ref{sec:infinite_volume_const} contain the graphical construction that is the key ingredient in the perfect simulation scheme as well in the proofs which are presented in Sections \ref{sec:perfect_simulation} and  \ref{sec:proofs} respectively. Conclusions and future works are discussed in Section \ref{sec:conclusion}.

\section{Definitions and notation}
\label{sec:def}

Let $V \subset \Zd{d}$ be a finite (or infinite) set and define the set of all simple graphs with vertex set $V$ by
\[\X^V=\{0,1\}^{E_V}\,,\]
where $E_V=\{\{i,j\} : i,j \in V, i\neq j\}$. \\
{\it Notation:} We denote $\X^{\Zd{d}}$ by $\X$ and  $E_{\Zd{d}}$ by $E$. We shall use $\g x \in \X^V$ to denote a graph where $x(\{i,j\})$ is $1$ if there exists an edge between $i$ and $j$ and $0$ otherwise, for $\{i,j\}\in E_V$. For simplicity, we write $x_{ij}$ for $x(\{i,j\})$. \\

Define $\g x^1_{ij}$ as the graph
which coincides with $\g x$ for all edges other than $\{i,j\}$ and $ x^1_{ij}=1$. In the same way, $\g x^0_{ij}$ is the graph
which coincides with $\g x$ for all edges other than $\{i,j\}$ and $ x^0_{ij}=0$. 

Define the length of an edge $\{i,j\}\in E$ by 
$$L(i,j)=\norm{i-j}\,,$$ 
where $\norm{\cdot}$ denotes the 1-norm. For $I, J \subset \Zd{d}$, define the distance
\begin{equation}
d(I,J) = \min\left\lbrace\,  L(i,j) : i\in I,\, j\in J, i\neq j\, \right\rbrace
\end{equation}
with the convention $d(I,I)=0$. We let $|G|_{d}$ denote the completion of the corresponding metric space as our topology on $\X^V$.

For any $i \in \Zd{d}$ denote by
\[B^i_k=\{\, j\in \Zd{d}, j\neq i : L(i,j) \leq k  \,\}\,\]
the ball of radius $k$ centered at $i$.

Let $\dg{i}{}(\g x)$ be the degree of vertex $i$ in the graph $\g x$ and let $\dg{i}{k}(\g x)$ be the degree of vertex $i$ restricted to the box $B_k^i$, that is, 
\begin{equation}\label{def:restricted_degree}
\dg{i}{k}(\g x) = \sum\limits_{j\in B_k^i}x_{ij}\,.
\end{equation}

For any $V\subset\Zd{d}$ and $f:\X^V \rightarrow \mathbb{R}$, define the vertex support of $f$, $\suppv{f}\subset \Zd{d}$, as the smallest set of vertices that fully determines the function $f$,  that is, $f$ depends on edges with one endpoint in $\suppv{f}$. More precisely, $\suppv{f}$  is uniquely determined by the following conditions
\begin{itemize}
\item $f(\g x)=f(\g x')$ whenever $\g{x}|_{\suppv{f}} =\g{x}'|_{\suppv{f}}$,
\item If $J\subset \Zd{d}$ is any other finite vertex set for which $f(\g x)=f(\g x')$ whenever $\g{x}|_{J} = \g{x}'|_{J}$, then $\suppv{f} \subset J$\,.
\end{itemize}
where $\g{x}|_{\Y}$ is the restricted graph to the set $\Y\subset \Zd{d}$ given by $\{\,x_{ij} : i,j \in \Y, i\neq j\}$. For instance, consider the function $f(\g x)=\ind\{x_{ij}x_{jk}x_{ik}=1\}$ that indicates the presence of a triangle between vertices $i,j$ and $k$, then $\suppv{f}=\{i,j,k\}$. In the case that $f$ represents the degree of vertex $i$ in the box $B^k_i$, $f(\g x) = d^k_i(\g x)$, the set $\suppv{f}$ is given by $B^i_k$.

\subsection{Spatial Gibbs Random Graphs}

The random graph model considered in this work describes an interplay between the sufficient statistics of the graph and the underlying space. We focus on random graphs with vertex set given by subsets of $\Zd{d}$.
We consider a model that penalizes connections between distant nodes in such a way that the Gibbs distribution defined on $V$  favors graphs with short edges.  Inspired by the Exponential Random Graph Model \citep{frank1986markov}, our model also penalizes edges through a sufficient statistic which is a function $F_V$ of the whole graph. In this way, for $\g{x} \in \X^V$ we define the following Hamiltonian 


\begin{equation}
\label{eq:hamiltonian_general}
H_{V}(\g x) = \sum\limits_{\{i,j\}\in E_V}L(i,j)x_{ij} +  F_V(\g x)\,.
\end{equation}

For each fixed $\beta \in \mathbb{R}^{+}$, the finite volume Gibbs distribution is given by
\begin{equation}\label{eq:def-gibbs}
\mu_{_{V}}(\g x) = \dfrac{\exp\{-\beta H_{V}(\g x)\}}{Z_V(\beta)}\,,
\end{equation}
where $Z_V(\beta)$ is the normalizing constant. \\

\noindent {\it Assumptions:} In this paper, at Equation \eqref{eq:hamiltonian_general}, we consider functions $F_V: \X^V\to\mathbb{R}$ such that 
\begin{enumerate}[label=(\subscript{A}{{\arabic*}})]
\item\label{assumption1} $F_V(\g x^1_{ij}) -F_V(\g x^0_{ij})$ only depends on edges that are connected with vertex $i$ or $j;$
\item\label{assumption2} there exists a finite constant $M$, $-1 < M < \infty$,  (which does not depend on $V$) such that
\begin{equation}\label{def:constant_M}
 M \le \min\limits_{\g x\in \X^V}\min\limits_{\{i,j\}E_V}(F_V(\g x^1_{ij}) -F_V(\g x^0_{ij}) )\,,
\end{equation}
for all finite $V$\,.
\item\label{assumption3} If $V \subset V'$ then $F_{V'}(\g{x}|_{V})=F_{V}(\g{x}|_{V})$.
\\ 
\end{enumerate}

Notice that the RHS of \eqref{def:constant_M} is always finite if $V$ is finite.


\begin{example} \label{ex:ferrari} [\cite{ferrari2010gibbs}] Consider a Gibbs measure that favors graphs with short edges, few vertices with degree zero and few vertices with degree greater or equal to $2$.  To this end,  define the function $F_V$ by
\begin{equation}\label{eq:function_F_ferrari}
F_V(\g x) = \sum\limits_{i\in V}\phi_i(\g{x})
\end{equation}
where
\begin{equation}
\phi_i(\g{x})=\left\{
\begin{array}{cl}
h_0, & \mbox{if } d_i(\g{x})=0\\
0, & \mbox{if } d_i(\g{x})=1 \\
h_1\binom{d_i(\g{x})}{2} & \mbox{if } d_i(\g{x})\geq 2
\end{array}
\right.
\end{equation}
where $0 < h_0 < h_1$ are fixed parameters. 

In this particular case,
\[
F_V(\g x^1_{ij}) -F_V(\g x^0_{ij}) = \phi_i(\g x^1_{ij}) - \phi_i(\g x^0_{ij} )+\phi_j(\g x^1_{ij}) - \phi_j(\g x^0_{ij})
\] 
and
\[\phi_i(\g x^1_{ij}) - \phi_i(\g x^0_{ij} )=\left\{
\begin{array}{cl}
-h_0, & \mbox{if } d_i(\g{x})=0\\
-h_0 & \mbox{if } d_i(\g{x})=1 \mbox{ and } x_{ij}=1\\
h_1, & \mbox{if } d_i(\g{x})=1 \mbox{ and } x_{ij}=0\\
h_1, & \mbox{if } d_i(\g{x})=2 \mbox{ and } x_{ij}=1\\
h_1\left(\binom{d_i(\g{x})}{2} - \binom{d_i(\g{x})-1}{2} \right) & \mbox{if } d_i(\g{x}) > 2 \mbox{ and } x_{ij}=1\\
h_1\left(\binom{d_i(\g{x})+1}{2} - \binom{d_i(\g{x})}{2} \right) & \mbox{if } d_i(\g{x})\geq 2 \mbox{ and } x_{ij}=0
\end{array}
\right.
\]
therefore, $M=-2h_0$.

\end{example}

\begin{example} \label{ex:stars} 
For the model that gives more weight to graphs with short edges and penalizes 2-stars, the Hamiltonian can be written by 

\begin{equation}\label{def:hamiltonian_2stars}
H_{V}(\g x) = \sum\limits_{\{i,j\}\in E_V}L(i,j)x_{ij} +  \frac{1}{2}\sum\limits_{i\in V}\sum\limits_{\substack{j,k \in V\\ j\neq  k}}L(i,j)L(j,k)x_{ik}x_{jk}\,.
\end{equation}

In this case, as well as for the models that penalize statistics of the graph such as k-stars and triangles, the constant $M$, given by \eqref{def:constant_M}, is equal to $0$. 
\end{example}

\subsection{Dependent graph process}

For a finite or infinite set $V \subset \Zd{d}$ and a real continuous function $f$ on $\X^V$, we define a Markov process on $\X^V$ for which the generator of the process is defined by
\begin{equation}\label{eq:generator_gibbs}
\begin{split}
\bold{A}^Vf(\g x) &= \sum\limits_{ \{i,j\}\in E_V }e^{-\beta L(i,j)-\beta M}\ind\{x_{ij}=0\}Q(\{i,j\}\,|\,\g x)[f(\g x^1_{ij}) - f(\g x)] \\
& \qquad+ \sum\limits_{\substack{ \{i,j\}\in E_V}} x_{ij}[f(\g x^0_{ij}) - f(\g x)]
\end{split}
\end{equation}
where
\begin{equation}\label{eq:prob_M}
Q(\{i,j\}\,|\,\g x) = \exp\left\lbrace -\beta(F_V(\g x^1_{ij}) -F_V(\g x))+\beta M\right\rbrace\,.
\end{equation}

It is worth noting that by the definition of $M$ in \eqref{def:constant_M} we have that $0\leq Q(\{i,j\}\,|\,\g x) \leq 1 $, for all $\{i,j\} \in E_V$ and all $\g x\in \X_V$. 

The process defined above has the following dynamics: when the current graph is $\g x$, the edge $\{i,j\}$ attempts to be born with rate $e^{-\beta L(i,j)-\beta M}$ and it is added in the graph with probability $\Q(\{i,j\}\,|\,\g x)$ if it is not already in the graph. An edge $\{i,j\}$ belonging to the current graph is removed at rate $1$. \\

\begin{example*} {\bf \ref{ex:stars}(cont.):} For the particular 2-stars model defined by the Hamiltonian \eqref{def:hamiltonian_2stars}, the generator of the dependent graph process is given by
\begin{equation}
\begin{split}
\bold{A}^Vf(\g x) &= \sum\limits_{ \{i,j\}\in E_V }e^{-\beta L(i,j)}\ind\{x_{ij}=0\}Q(\{i,j\}\,|\,\g x)[f(\g x^1_{ij}) - f(\g x)] \\
& \qquad+ \sum\limits_{\substack{ \{i,j\}\in E_V}} x_{ij}[f(\g x^0_{ij}) - f(\g x)]
\end{split}
\end{equation}
where
\begin{equation}
Q(\{i,j\}\,|\,\g x) = \exp\left\lbrace -\frac{\beta}{2}\left(\sum\limits_{k\in V}L(i,j)L(i,k)x_{ik} +\sum\limits_{k'\in V}L(i,j)L(j,k')x_{jk'} \right) \right\rbrace\,.
\end{equation}
\end{example*} 

For $V\subset \Zd{d}$ finite, it is easy to see that the invariant measure of the process defined above is $\mu_{_V}$ given by \eqref{eq:def-gibbs}. For $V$ infinite, the existence of the Markov process \eqref{eq:generator_gibbs} is not a priori clear, and needs to be proved.

\section{Main Results} \label{sec:main}

For $\beta\in \mathbb{R}^{+}$, define 
\begin{equation}\label{def:alpha}
\alpha(\beta) = 2^{d+1}e^{-\beta M}\left( \frac{1}{(1-e^{-\beta})^d}-1\right)\,
\end{equation}
and let 
\begin{equation}
\beta^{\ast} = \inf\{\beta > 0 : \alpha(\beta) \leq 1\}\,.
\end{equation}

Our first result guarantees the existence of at least one invariant measure of the process in the case of graphs with infinite number of vertices.

\begin{theorem}\label{theo:existence}
If $\alpha(\beta)<\infty$, then for any infinite $V \subset \Zd{d}$ the Markov process with generator $A^V$ exists and admits at least one invariant measure.
\end{theorem}

\indent Theorem \ref{theo:uniqueness} guarantees the existence and uniqueness of an infinite volume distribution $\mu$ as a limit along sub-sequences of $\mu_{_V}$, as $V\rightarrow \Zd{d}$. Furthermore,  we show that under $\mu$ all graphs have finite vertex degree with probability $1$.

\begin{theorem}\label{theo:uniqueness}
If $\beta > \beta^{\ast}$, then the following statements hold:
\begin{enumerate}
\item For any $V \subset \Zd{d}$ there exists a unique process $\eta^{V}_{\,t}$ with generator $A^V$. The process has a unique invariant measure given by $\mu_{_V}$. For $V$ finite, $\mu_{_V}$ is the measure defined by \eqref{eq:def-gibbs}. For $V=\Zd{d}$, we denote $\mu_{_{\Zd{d}}}$ by $\mu$.
\item\textit{(Weak convergence)} As $V \rightarrow \Zd{d}$, $\mu_{_V}$ converges weakly to $\mu$ and $\mu$ is concentrated on
 \[\{\g x \in \X : \dg{i}{}(\g x) < \infty, \mbox{ for all } i \in \Zd{d}\}\,.\]
\end{enumerate}
\end{theorem}

\begin{example*} {\bf \ref{ex:ferrari}(cont.): } Although the authors have proved the existence of the infinite measure $\mu$ for this model, the uniqueness of this measure has not been addressed by them.
Thus, as complement of their work, we have by Theorem \ref{theo:uniqueness}-(1), that the uniqueness of $\mu$ is guaranteed whenever $0 < h_0 < 1/2$ and $\beta > \beta^{\ast}$. \\
\end{example*}

\begin{example*} {\bf \ref{ex:stars} (cont.): } For the particular 2-stars model, the value of $\alpha(\beta)$ in function of $\beta$ and the dimension $d$ is given in Figure \ref{fig:alpha_function}. \\

\begin{figure}[h]
\centering
\includegraphics[scale=0.55]{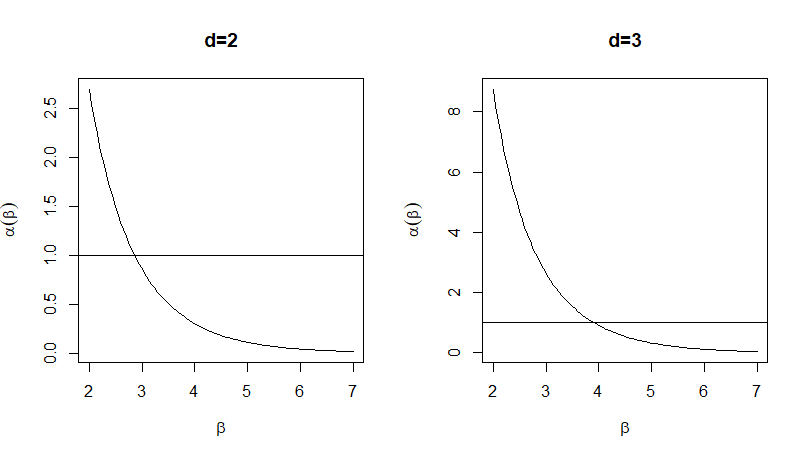}
\caption{Function $\alpha(\beta)$ for the model that penalizes 2-stars in the graph. }
\label{fig:alpha_function}
\end{figure}
\end{example*}

Once the weak convergence of the finite measure $\mu_{_V}$ to the infinite measure $\mu$ is guaranteed, Theorem \ref{theo_space_bounds} states that the convergence is exponentially fast in space. 

\begin{theorem}[\textbf{Exponential space convergence}]\label{theo_space_bounds}
Let $V$ be a finite subset of $\Zd{d}$ and assume that  $\beta > \beta^{\ast}$ and $M>0$. If $f$ is a measurable function with $\suppv{f}\subset V$, then
\begin{equation}
\left|\mu f - \mu_{_{V}}f \right|\, \leq \, \left(\dfrac{2e^{-(\beta-\til\beta)M}\alpha(\til{\beta})||f||_{\infty}}{1-e^{-(\beta-\til\beta)M}\alpha(\til{\beta})}\right)\sum\limits_{i \in \suppv{f}}e^{-(\beta-\til\beta)d(\{i\},V^c)}\,,
\end{equation} 
for any $\til{\beta} \in (\beta^{\ast},\beta)$.
\end{theorem}

Example \ref{ex:restricted_degree} below illustrate how Theorem \ref{theo_space_bounds} can be applied in order to better understand the relation between the infinite and finite measures through some characteristics of the graph. In particular, we analyze
the degree of a vertex restricted to a finite box with respect to the infinite and finite measures.

\begin{example}[\textit{Expectation of the restricted degree}]\label{ex:restricted_degree}
For any $i \in \Zd{2}$ and $k,l \in \mathbb{N}$, with $k \leq l$, set $f(\g x)=\dg{i}{k}(\g x )$, i.e, the degree of vertex $i$ restricted to the box $B_k^i$ as defined in \eqref{def:restricted_degree}.  Since $\suppe{v}=B^i_k$ and we have
\[||f||_{\infty} = |B^i_k| = 2k(k+1)\]
and
\begin{equation}\label{eq:col_restricted_1}
\begin{split}
\sum\limits_{j\in B^i_k}e^{-(\beta-\til\beta)d(\{j\},(B_l^i)^c)} &= \sum\limits_{s=1}^k 4se^{-(\beta-\tilde{\beta})(l+1-s)}\,.
\end{split}
\end{equation}
By Theorem~\ref{theo_space_bounds} we get
\begin{equation}\label{eq:example_rest_degree}
\left|\mu f - \mu_{_{B_l^i}}f \right|\, \leq \, \left(\dfrac{4k(k+1)e^{-(\beta-\til\beta)M}\alpha(\til{\beta})e^{-(\beta-\tilde{\beta})(l+1)}}{1-e^{-(\beta-\til\beta)M}\alpha(\til{\beta})}\right)\sum\limits_{s=1}^k 4se^{(\beta-\tilde{\beta})s}\,.
\end{equation}\\
Notice that if $l\to \infty$, then the right-hand side of \eqref{eq:example_rest_degree} goes to $0$.
\end{example}

Theorem \ref{theo:mixing} states the mixing property for the measure $\mu_{_V}$ for $V$ finite or infinite.

\begin{theorem}[\textbf{Exponential mixing}]\label{theo:mixing}
Let $V$ be a subset of $\,\Zd{d}$ (finite or infinite) and assume that  $\beta > \beta^{\ast}$ and $M>0$. If $f$ and $g$ are measurable functions with $\suppv{f},\suppv{g}\subset V$, then
\begin{equation}
\begin{split}
&\left|\mu_{_V} (fg) - \mu_{_{V}}f\mu_{_{V}}g \right|\\
& \qquad \leq \, 2\left(\dfrac{e^{-(\beta-\til\beta)M}\alpha(\til{\beta})}{1-e^{-(\beta-\til\beta)M}\alpha(\til{\beta})}\right)^2||f||_{\infty}||g||_{\infty}\sum\limits_{ \substack{ i \in \suppv{f} \\  j \in \suppv{g}}}L(i,j)e^{-(\beta-\til\beta)L(i,j)}\,,
\end{split}
\end{equation} 
for any $\til{\beta} \in (\beta^{\ast},\beta)$.

\end{theorem}

\begin{example}
For any $i,j,k,l \in \Zd{d}$ and $V\subset \Zd{d}$ set $f(\g x)=\ind\{x_{ij}=1 \}$ and $g(\g x)=\ind\{x_{kl}=1 \}$. Since $||f||_{\infty}=||g||_{\infty}=1$, $\suppv{f}=\{i,j\}$ and $\suppv{g}=\{k,l\}$  using Theorem \ref{theo:mixing}
\begin{equation}\label{eq:example_mixing1}
\begin{split}
&\left|\mu_{_V}( \g x : x_{ij}=1,x_{kl}=1) - \mu_{_V}(\g x : x_{ij}=1)\mu_{_V}(\g x : x_{kl}=1)\right|\, \\
&\qquad\leq 2\left(\dfrac{e^{-(\beta-\til\beta)M}\alpha(\til{\beta})}{1-e^{-(\beta-\til\beta)M}\alpha(\til{\beta})}\right)^2\sum\limits_{ \substack{ i \in \suppv{f} \\  j \in \suppv{g}}}L(i,j)e^{-(\beta-\til\beta)L(i,j)}\,.
\end{split}
\end{equation} 
Observe that when the edges $\{i,j\}$ and $\{k,l\}$ are far apart from each other, that is, $d(\{i,j\},\{k,l\})\to \infty$, the right-hand side of \eqref{eq:example_mixing1} goes to $0$.
\end{example}

The mixing property stated in Theorem \ref{theo:mixing} combined with the central limit theorem for stationary mixing random fields in \cite{bolthausen1982central} allow us to establish a generalization of a central limit theorem for functions of graphs with finite vertex support.

\begin{theorem}[\textbf{Central limit theorem}]\label{theo:central_limit}
Let $f$ be a measurable function on $\X$ with finite vertex support such that $\mu(|f|^{2+\delta}) < \infty$, for some $\delta>0$. Let $\tau_i$ be a translation by $i$ and assume that $\beta>\beta^{*}$ and \\
\[\sigma^2=\sum\limits_{i\in \Zd{d}}(\,\mu(f\tau_if)-\mu(f)\mu(\tau_if)\,)>0.\] Then, 
\begin{equation}\label{eq:theo_central_1}
\sum\limits_{i\in \Zd{d}} \mid\mu(f\tau_if)-\mu(f)\mu(\tau_if) \mid < \infty
\end{equation}
and 
\begin{equation}\label{eq:theo_central_2}
\dfrac{1}{\sqrt{|V|}}\left( \sum\limits_{i\in V}( \tau_i f- \mu_{_V} (\tau_i f) )  \right)\,\, \mathrel{\mathop{\Longrightarrow}^{\mathcal{D}}_{V\to \Zd{d}}}\,\, \mathcal{N}(0,\sigma^2)\,.
\end{equation}
\vspace{0.1cm}
\end{theorem}

\begin{example}
Define the function $ f(\g x) = \ind\{\dg{0}{k}(\g x) \geq 1\}$, where $\dg{0}{k}(\g x)$ is the degree of the vertex located in the origin of $\Zd{d}$. Define its translation $\tau_i f(\g x) = \ind\{\dg{i}{k}(\g x) \geq 1\}$, for fixed $k\in \mathbb{N}$. Assuming that $\beta > \beta^{\ast}$ and define
\[\sigma^2 = \sum\limits_{i\in \Zd{d}} \mu(\g{x} :\dg{0}{k}(\g x) \geq 1, \dg{i}{k}(\g x) \geq 1 ) -\mu(\g{x} :\dg{0}{k}(\g x) \geq 1)\mu(\g{x} :\dg{i}{k}(\g x) \geq 1) > 0\,, \]
then we have that
\[\dfrac{1}{\sqrt{|V|}} \left(\sum\limits_{i\in V}\left(\,\ind\{\dg{i}{k}(\g x) \geq 1\}  -\mu_{_V}(\g x : \dg{i}{k}(\g x) \geq 1)\,\right)\right)\,\, \mathrel{\mathop{\Longrightarrow}^{\mathcal{D}}_{V\to \Zd{d}}}\,\, \mathcal{N}(0,\sigma^2)\,.\]
\end{example}

\section{Graphical Representation}\label{sec:graphical_representation}

In this section we present the graphical construction of the birth and death process inspired by \cite{fernandez2001loss} which will be the key ingredient to prove all the results stated before as well as to construct a perfect simulation scheme. The construction and the proofs are similar to those  in \cite{fernandez2001loss}, therefore we will omit the proofs.

To each pair of vertices $\{i,j\}\in E$ we associate an independent marked Poisson process on $\mathbb{R}$ with rate $e^{-\beta L(i,j)-\beta M}$. Let $T_k^{ij}$ be the ordered occurrence times of the Poisson process such that $T^{ij}_0 < 0 < T_1^{ij}$. To each occurrence time $T_k^{ij}$ we associate an independent mark $S^{ij}_k$ exponentially distributed with mean 1 and an independent mark $U^{ij}_k$ uniformly distributed on $(0,1)$. In a nutshell, an edge $\{i,j\}$ is born at the random times $T^{ij}_k$ and it survives $S^{ij}_k$ time units.

We define the random family of marked Poisson process by
\[\rec = \{ \, \{ (\, \{i,j\} ,  T^{ij}_k , S^{ij}_k , U^{ij}_k\,) : k \in \mathbb{Z} \}\,  :\{i,j\}\in E \}.\]
Each quartet  $R= (\{i,j\} ,  T^{ij}_k , S^{ij}_k , U^{ij}_k) \in \mathbf{R} $ can be represent by a marked rectangle $(\{i,j\} \times [T^{ij}_k , T^{ij}_k +S^{ij}_k], U^{ij}_k )$ with \textit{basis} $\{i,j\}$, \textit{birth time} $T^{ij}_k$, \textit{lifetime} $S^{ij}_k$ and mark $U^{ij}_k$. In this way, for a rectangle $R=(\{i,j\}, t, s, u )$ we denote $\basis{R}=\{i,j\}$ , $\birth{R}=t$, $\death{R}=t+s $, $\life{R}=[t, t+s]$ and $\flag{R}=u$.

For an initial graph $\g {z}$, it is associated an independent random initial life time $S_0^{ij}$, exponentially distributed with mean $1$, and an independent uniform mark $U_0^{ij}$ on $(0,1)$ for each edge $\{i,j\}$ in graph $\g {z}$ ($z_{ij}=1$). Define the set of initial rectangles
\begin{equation}\label{eq:initial_rectangles}
\reci{z} = \{ \,(\, \{i,j\} ,  0, S^{ij}_0 , U^{ij}_0\,) :  \mbox{for } \{i,j\}\in E \text{ such that } z_{ij}=1 \}\,.
\end{equation}

For $s,t \in \mathbb{R}$, $s< t$, define the set of rectangles born on the time interval $[s,t]$ by
\[ \rect{s}{t}  =\{ \, R \in \rec : \birth{R}\in [s,t] \}\,. \]

For the model defined by \eqref{eq:def-gibbs}, Assumption \ref{assumption1} guarantees that the existence of the edge $\{i,j\}$ in the graph depends on the edges that are connected to vertex $i$ or $j$. In general, we say that there exists a \textit{dependence relation} between two edges in the graph if they share a common vertex. We define this dependence relation ($\sim$)
between edges by
\[ \{i,j\}\sim\{k,l\} \qquad\mbox{ if }\quad \{i,j\}\cap\{k,l\}\neq \emptyset \]
and between rectangles by
\begin{equation}\label{def:dependence}
 R \sim R' \qquad\mbox{ if }\quad \basis{R}\cap\basis{R'}\neq\emptyset \quad\mbox{ and } \quad \life{R}\cap\life{R'}\neq \emptyset \,.
\end{equation}

In the next sections, we consider the probability space given by the product of the spaces generated by the rectangles $\rec$ and initial rectangles $\reci{\g z}$. We denote it by $(\Omega, \mathscr{F}, \P)$. We also write $\esp$ for the respective expectation.

\subsection{Construction of the independent multigraph process}\label{sec:multigraph_process}

To make the notation easier to follow we shall reserve the bold roman letters $\g{x},\g{z},\g{y}$  to represent graphs and the greek letters $\eta, \xi$ to represent the processes defined on graphs.

For $\g z \in \mathbb{N}^E$, define the process $(\free{\g z})_{t\geq 0}$ on $\mathbb{N}^E$ by
\begin{equation}
\free{\g z}(i,j)=\sum\limits_{R \in \rect{0}{t}\cup \reci{z}}\ind \{ \basis{R}=\{i,j\}, \life{R}\ni t  \}\,.
\end{equation}

The process described above is a product of independent birth-and-death process on $\mathbb{N}^E$ with initial graph $\g {z}$ whose generator is given by
\begin{equation}\label{eq:generator_indep_multi}
A^0f(\g x) = \sum\limits_{\{i,j\}\in E }e^{-\beta L(i,j)-\beta M}[f(\g x^1_{ij}) - f(\g x)] + \sum\limits_{\substack{ \{i,j\}\in E}} x_{ij}[f(\g x^0_{ij}) - f(\g x)]\,.
\end{equation}

In this ``free" process an edge is added in the graph every time it tries to be born.  Because of this lack of restriction, an edge is allowed to be added in the graph when it is already in the graph, giving rise to a \textit{multigraph} structure. In this case, $\free{\g z}(i,j)$ corresponds to the number of edges connecting $i$ and $j$ at time $t$.

The invariant and reversible measure for this process, denoted by $\mu^0$, is the product distribution on $\mathbb{N}^E$ for which the (marginal) number of multiple edges $\{i,j\}$ is given by a Poisson random variable with mean $e^{-\beta L(i,j)-\beta M}$, that is,
\begin{equation}
\mu^0(x_{ij}=k) = \dfrac{(e^{-\beta L(i,j)-\beta M})^k}{k!}\exp(e^{-\beta L(i,j)-\beta M})\,.
\end{equation}

In a nutshell, the invariant measure $\mu_0$ is defined on the set of multigraphs with independent edges. Because of this well defined structure, this measure will be called \textit{independent multigraph distribution}.

\subsection{Finite volume construction of the dependent process}\label{sec:finite_construction}

As defined in Section \ref{sec:multigraph_process}, the independent multigraph process is constructed using the graphical representation by rectangles introduced in Section \ref{sec:graphical_representation}.
In this section, we describe the cleaning operation that should be applied in the independent multigraph process in order to construct the process $(\dep{V}{\g x})_{t\geq0}$ on $\X^V$, for $V \subset \Zd{d}$ finite, with generator given by \eqref{eq:generator_gibbs}.

Let $\rec^V[0,t] = \{R\in \rec[0,t] : \basis{R}\subset E_V\}$. To construct the independent process on $\X^V$, for $V \subset \Zd{d}$ finite, we use the set of rectangles $\rec^V[0,t]$ and the set of initial rectangles $\rec^V(\g{x})$ associate with the initial graph $\g{x}\in \X^V$. 
To construct the dependent process with generator given by \eqref{eq:generator_gibbs}, some rectangles are erased from the set $\rec^V[0,t]\cup \rec^V(\g{x})$, using a cleaning procedure, resulting the set $\kept_{\g{x}}^{V}[0,t]$ of \textit{kept} rectangles at time $t$. The cleaning procedure used to decide which rectangles are erased or kept are described below.

At time $0$ we include all rectangles of $\rec^V(\g{x})$ in $\kept_{\g{x}}^{V}[0,t]$. Since $V$ is finite we can move forward ordering the birth and death marks as $0 < r_1 < r_2 < \cdots < r_N < t$.

We construct the process $\dep{V}{\g x}$ as following:

\begin{enumerate}
\item We set $\eta^{\,V,\,\g x}_{\,0}=\g{x}$.
\item Supose that $\eta^{\,V,\,\g x}_{\,r}$ is already defined, and that $r_{k-1} \leq r < r_{k}$. We set
\[\eta^{\,V,\,\g x}_{\,s}=\eta^{\,V,\,\g x}_{\,r}\, , \quad r \leq s < r_k\,. \]
if $r \geq r_N$, then
\[\eta^{\,V,\,\g x}_{\,s}=\eta^{\,V,\,\g x}_{\,r}\, ,\quad r \leq s < t\,. \]
\item If $r_k$ is a death time, that is, $r_k=l+s$ for some $R=(\{i,j\},l,s,u) \in \rec^V[0,t]$, then we delete the edge $\{i,j\}$ of the graph by setting, for all $\{m,n\}\in V$,
\[  \eta^{\,V,\,\g x}_{\, r_k}(m,n)=\begin{cases}
0, \text{ if } \{m,n\}=\{i,j\}\\
\eta^{\,V, \, \g x}_{\,r_{k-1}}(l,m), \text{ otherwise }
\end{cases}\]
Go back to step [2].
\item If $r_k$ is a birth time, that is, $r_k=l$ for some $R=(\{i,j\},l,s,u) \in \rec^V[0,t]$, then if
\begin{equation}\label{eq:test_kept}
\eta^{\,V,\,\g x}_{\,r_{k-1}}(i,j)=0 \text{ and } u < \Q(\{i,j\}\,|\,\eta^{\,V,\,\g x}_{\,r_{k-1}})
\end{equation}
we add the edge $\{i,j\}$ in the graph by setting 
\[\eta^{\,V,\,\g x}_{\, r_k}(m,n)=\begin{cases}
1, \text{ if } \{m,n\}=\{i,j\}\\
\eta^{\,V,\,\g x}_{\, r_{k-1}}(m,n), \text{ otherwise }
\end{cases}\]
and we keep the rectangle $R$, that is, $R$ is added in the set of kept rectangles $\kept_{\g{x}}^{V}[0,t]$.
In either case, set $\eta^{\,V,\,\g x}_{\, r_k}=\eta^{\,V,\,\g x}_{\, r_{k-1}}$ go back to step [2].
\end{enumerate}

We can also construct the process $\eta_{\,t}^{V,\, \g x}$ described above directly from the set of kept rectangles. To do this, we first generate rectangles by running the independent multigraph process from $0$ to $t$ with initial graph $\g x \in \X^V$. After that, we decide which rectangles are kept using successively the test given by \eqref{eq:test_kept}. Basically, the test $\eqref{eq:test_kept}$ does not allow multiple edges in the graph and it only adds a non-existent edge in the graph with probability $\Q(\cdot|\cdot)$ given by \eqref{eq:prob_M}.  Using directly the kept rectangles, the process $\dep{V}{\g x}$ on $\X^V$ is defined by
\begin{equation}\label{eq:dep_const_kept}
\dep{V}{\g x}(i,j) = \ind \left\lbrace \{i,j\}\in \{\basis{R} : R\in\kept_{\g{x}}^{V}[0,t] ,\,  \life{R}\ni t  \}\,\right\rbrace \,.
\end{equation}

We show in Theorem \ref{theo:clan_finite} that $\dep{V}{\g x}$ has generator $\bold{A}^V$ given by \eqref{eq:generator_gibbs} restricting the sums to the set of pairs of vertices contained in $E_V$. Since $\mu_{_V}$ is reversible for this process and we have an irreducible Markov process with a finite state space $\X^V$, $\dep{V}{\g x}$ converges in distribution to $\mu_{_V}$ for any initial graph $\g x \in \X^V$. This implies that $\mu_{_V}$ is the unique invariant measure for this process.

Set two initial graphs $\g x\in \X^V$ and $\g z \in \mathbb{N}^{E_V}$ such that $x_{ij} \leq z_{ij}$, for all $\{i,j\}\in E_V$. 
We construct the process $\dep{V}{\g x}$ and $\free{\g z}$ using the same set $\rec$ and the same set of initial rectangles for common edges in $\g x$ and $\g z$. Since in the independent multigraph process $\free{\g z}$ all rectangles are kept we have that

\begin{equation}\label{eq:relacao_free_dep}
\dep{V}{\g x}(i,j) \, \leq \,  \free{\g z}(i,j) \quad \mbox{for all } \{i,j\}\in E_V\,.
\end{equation}

The construction given by \eqref{eq:dep_const_kept} can be done in a stationary way for $t \in \mathbb{R}$. Indeed, since $E_V$ is a finite set, there exists a sequence of random times $\tau_i$ with $\tau_i \to \pm\infty$ as $i \to \pm\infty$, such that $\xi_{\tau_i}$ corresponds to the empty graph, that is $\xi_{\,\tau_i}(i,j)=0$, for all $\{i,j\}\in E_V$. In other words, in each $\tau_i$ no rectangle is alive.  Thus, we can construct the set of kept rectangles independently in each random intervals $[\tau_i,\tau_{i+1})$ using the rectangles of $\rec[\tau_i,\tau_{i+1}] $ and forgetting the set of initial rectangles. Let us denote by $\kept^{V}$ the resulting set of kept rectangles and $\eta^{\,V}_{\,t}$ the process defined as in \eqref{eq:dep_const_kept}. By construction, $\kept^V$ has a time translation-invariant distribution. The process $\eta^{\,V}_{\,t}$ has generator $\bold{A}^V$ given by \eqref{eq:generator_gibbs} and distribution independent of $t$ given by $\mu_{_V}$. This implies that, for any measurable function $f$ and any $t\in \mathbb{R}$,
\begin{equation}\label{eq:exp_prob_finite}
\mu_{_V} f = \esp\left[  f(\eta^{\,V}_{\,t} )\right] \,.
\end{equation}

Taking $f(\g x) = x_{ij}$, for $\{i,j\}\in E_V$, we have that
\begin{equation}
\mu_{_V}(\, \eta^{\,V}(i,j) = 1 \,)\, \leq \, e^{-\beta L(i,j)-\beta M}\,,
\end{equation}
since $\eta^{\,V}_{\,t}(i,j) \leq \free{}(i,j)$ for all $\{i,j\}\in E_V$ and $\free{}(i,j)$ has Poisson distribution with mean $e^{-\beta L(i,j)-\beta M}$.

\section{Infinite volume construction of the dependent model}
\label{sec:infinite_volume_const}

The finite volume procedure described in Section \ref{sec:finite_construction} cannot be directly applied to construct the process for infinite $V$, since in this case we can not order the birth marks because there is not a first Poisson mark.

However, in the finite construction, in order to decide whether a rectangle $R$ is kept at time $t$,  it is necessary to look at the set of rectangles that were born before $R$, are alive at the birth time of $R$ and whose basis intersects $R$. This describes the relation of ``being an ancentor of'' as defined in the works \cite{fernandez2001loss} and \cite{ferrari2002perfect}.

For a point $(i,t)\in \Zd{d}\times\mathbb{R}$ define the set of rectangles containing vertex $i$ that are alive at time $t$ by
\begin{equation}
\begin{split}
\A^{i,t}_1 &= \{R \in\rec : \basis{R}\ni i , \life{R}\ni t  \}\,.
\end{split}
\end{equation}

For a rectangle $R$, define the set of ancestors of $R$ as the set of rectangles born before $R$ that have the dependence relation with $R$, that is
\begin{equation}
\A^R_1 = \{R'\in \rec : R\sim R', \birth{R'} < \birth{R}\}\,.
\end{equation}

The \textit{n}th generation of ancestors of the rectangle $R$ is defined recursively by
\begin{equation}
\A^R_n=\{R'' \in \rec : R'' \in A_1^{R'} \mbox{ for some } R'\in A_{n-1}^R\}
\end{equation}
and, for a point $(i,t)\in \Zd{d}\times\mathbb{R}$ it is defined by
\begin{equation}
\begin{split}
\A^{i,t}_{n}&=\{R'' \in \rec : R'' \in A_1^{R'} \mbox{ for some } R'\in A_{n-1}^{i,t}\}\,.
\end{split}
\end{equation}

We define the \textit{clan of ancestors} of the vertex $i$ at time $t$ as the collection of all generations of ancestors of the point $(i,t)$ and denote it by
\begin{equation}\label{def:clan_vertice}
\A^{i,t} = \bigcup\limits_{n\geq 1}\A_n^{i,t} \,.
\end{equation}

The relation ``being an ancestor of'' gives rise to a model of \textit{backward oriented percolation}. We say that there is backward oriented percolation in $\rec$ if there exists $(i,t)$ such that $\A^{i,t}_n\neq \emptyset$ for all n; that is, there exists $(i,t)$ with infinitely many generations of ancestors. 
Theorem \ref{theo:clan_finite} states that in a finite time interval, the existence of the graph process defined on $\Zd{d}$ is guaranteed as long as all rectangles associate to each vertex $i$ of the graph have not a infinite number of ancestors. In other words, the existence of edges with one end given by the vertex $i$ depends only on a finite set of vertices of graph.
\begin{theorem}\label{theo:clan_finite}
\begin{enumerate}
\item If, with probability 1, $\A^{i,t}\cap \rec[0,t]$ is finite for any $i\in \Zd{d}$ and $t \geq 0$, then for any (possible infinite) $V \subset \Zd{d}$, the process with generator $A^V$ is well defined for any initial graph $\g x \in \X^V$ and has at least one invariant measure $\mu^V$.
\item If, with probability 1, there is no backward oriented percolation in $\rec$, then the process with generator $A$ can be constructed in $(-\infty,\infty)$ in such a way that the marginal distribution of $\eta_t$ is invariant.
\end{enumerate}
\end{theorem}

The existence of the invariant measure of the process with generator $A$, as stated in Theorem \ref{theo:uniqueness}-(1), follows directly from Theorem \ref{theo:clan_finite} (the complete proof of Theorem \ref{theo:uniqueness} is given in Section \ref{sec:proofs}). 
In this way, we shall denote by $\mu$ the marginal distribution of the process $\eta_{\,t}$ described in Theorem \ref{theo:clan_finite}-(2). As in the finite case,
\begin{equation}\label{eq:ineq_free_dep}
\eta_{\,t}(i,j)\, \leq \, \free{}{(i,j)}
\end{equation}
for all $\{i,j\}\in E$. Analogously to the finite case, 
\begin{equation}\label{eq:ineq_free_dep_exp}
\mu\{\,\eta_{\,t}(i,j)=1\,\} = \esp\, \eta_{\,t}(i,j)\, \leq \, \esp\, \free{}{(i,j)}=e^{-\beta L(i,j)-\beta M}\,.
\end{equation}

By \eqref{eq:process_t}, we have, for any $t\in\mathbb{R}$ and any measurable function $f$, that
\begin{equation}\label{eq:exp_process_inf}
\mu f = \esp f(\eta_{\,t})\,.
\end{equation}

\section{Perfect simulation of $\mu$}\label{sec:perfect_simulation}

The graphical construction described in Section \ref{sec:graphical_representation} induces directly the construction of a perfect simulation algorithm to sample a subgraph from the measure $\mu$.
The idea behind the simulation algorithm involves the construction of the process $\eta_{\,t}$ as described in the proof of Theorem \ref{theo:clan_finite}-(2). In a nutshell, since we assume no backward oriented percolation in $\rec$, the set of kept rectangles $\kept$ can be constructed clan by clan and the process is defined by
\begin{equation}
\eta_{\,t}(i,j) = \ind \left\lbrace \{i,j\}\in \{\basis{R} : R\in\kept ,\,  \life{R}\ni t  \}\,\right\rbrace \,.
\end{equation}

By Theorem \ref{theo:clan_finite}-(2), $\eta_t$ can be constructed in $(-\infty,\infty)$ in such a way that the marginal distribution of $\eta_t$ is $\mu$. Thus, we focus on the construction of the process at time $0$. 

For a finite set of vertices $V \subset \Zd{d}$, this construction involves the set of kept rectangles at time $0$, that can be obtained through the clan of ancestors of all vertices in $V$.
Thus, we obtain the clan of ancestors $\A^{V,\,0}$ such that all rectangles belonging to the clan have basis in $E_{V}$ and they are alive at time $0$. Once we have constructed the clan of ancestors, we only need to apply the cleaning procedure as described in the finite case through the test \eqref{eq:test_kept}.

The algorithms described in this section are based in a non-homogeneous time-backwards construction of the clan of ancestors based in the results proven in the works \cite{fernandez2001loss} and \cite{ferrari2002perfect}. They proved that the clan of ancestors can be obtained coming back in time and generating births of the ancestors with a rate given by the density of the independent multigraph process multiplied by an exponential time factor ensuring that the ancestor has a lifetime large enough to be an ancestor.

For a finite set of vertices $V$ and a finite set of rectangles $\g {H}$, we define the set of edges that are potential ancestors of $\g H$ and $V\times \{0\}$ by
\begin{equation}
\begin{split}
E(\g H, V) &= \{ \{i,j\}\in E : \{i,j\}\sim\basis{R}, \text{ for some } R\in \g H\,\}\\
&\quad \bigcup \{\,\{i,j\}\in E : \{i,j\}\cap V \neq \emptyset\,\}\,.
\end{split}
\end{equation}

For an edge $\{i,j\}\in E(\g H, V)$, define
\begin{equation}
TI(\g H, V, \{i,j\}) = \min\{ \birth{R} : R\in \g H, \basis{R}\sim \{i,j\} \}\,.
\end{equation}

By convention $\min\emptyset=0$. Observe that $TI(\g H, V, \{i,j\})\leq 0$.\\

The following result follows immediately from Theorem 2 in  \citep{ferrari2002perfect}  and we omit its proof. 

\begin{theorem} \label{theo:clan_process}
The clan $\A^{V,\,0}$ is the limit as $t\to \infty$ of a process $\g A_t$, defined by the initial condition $\g A_0=\emptyset$ and the evolution equation
\begin{equation}\label{eq:evolution_eq}
\begin{split}
&\Scale[0.9]{\esp\left( \dfrac{\mathrm{d}f(\g A_t)}{\mathrm{d}t} \mid \g A_s , 0 \leq s \leq t \right)} \\
& \quad \Scale[0.95]{= \sum\limits_{\{i,j\}\in E(\g A_t, V)}\int_{t - TI(\g A_t, V, \{i,j\})}^{\infty} \mathrm{d}s\, e^{-s}e^{-\beta L(i,j)-\beta M}[\,f(\g A_t \cup (\{i,j\},t,s)) - f(\g A_t)\,]}
\end{split}
\end{equation}
Here $f$ is an arbitrary function depending on a finite number of edges intersecting $V$.
\end{theorem}

Observe that $\g A_t$ is a monotone process ($\g A_t \subset \g A_{t+1}$) in which at time $t$ only rectangles with basis in $E(\g A_t,V)$ can be included. A rectangle born at time $-t$ is included if 
\begin{itemize}
\item its basis is dependent ($\sim$) of the basis of some rectangle born later and its lifespan reaches the birth time of such rectangle; or
\item its basis is independent of those of all rectangle born later, but it intersects $V$ and the rectangle survives up to time equal to zero. 
\end{itemize}

The observation described above and Theorem \ref{theo:clan_process} can be translated in the following algorithms.\\

\textit{Algorithm 1: Construction of the backward clan $\A^{V,\,0}$. }
\begin{changemargin}{1.5cm}{1.5cm} 
\begin{enumerate}
\item Set $l=0$ and $\tau_0=0$. Generate $S_0^{\,ij}$, for all $\{i,j\}\in E_V$, independent mean one exponential random variables. Set 
\[\g H =  \{ \,(\{\{i,j\},0,S_0^{\,i_kj_k}\} ): \{i,j\}\in E_V  \}\,. \]
\item For each $\{i,j\}\in E(\g H, V)$, generate an independent random variable $\tau(\{i,j\})$ such that
\[\P(\tau(\{i,j\}) > t) = 1 - \exp(-\nu_{ij}(t))\]
where
\[\nu_{ij}(t) = e^{-\beta L(i,j)-\beta M}e^{-t + TI(\g H, V, \{i,j\})}\ind\{t > \tau_{\,l}\}\,.\]
\item Set $l=l+1$ and $\tau_l=\min\{\tau(\{i,j\}) : \{i,j\}\in E(\g H, V)\}$
\begin{itemize}
\item If $\tau_{\,l} < \infty$, call $\{k,m\}$ be the edge such that $\tau(\{k,m\})=\tau_{\,l}$. Let
\[\g H = \g H\cup\{ \,(\{k,m\},-\tau_{\,l},\tau_{\,l}+TI(\g H, V, \{k,m\}) + S^l)\,\} \]
where $S^l$ is an exponential mean one random variable independent of everything else. Go back to (2).
\item If $\tau_{\,l} = \infty$ set $A^{V,\,0}=\g H$ and stop.
\end{itemize}
\end{enumerate}
\end{changemargin}

\textit{Algorithm 2: Construction of the kept rectangles using a cleaning procedure.}
\begin{changemargin}{1.5cm}{1.5cm} 
\begin{enumerate}
\item Start with $\g H = A^{V,\,0}$ and $\g K=\emptyset$
\item If $\g H = \emptyset$ go to (5). If not, order the rectangles of $\g H$ by time of birth. Let $R_1$ be the first of those rectangles and call $\{i,j\}$ its basis and $\tau_1$ its birth time. Let
\begin{equation}
\begin{split}
&\eta(k,l) = \ind\{ \{k,l\} \in \{ \basis{R} : R\in \g K,\\
&\qquad\qquad\quad \basis{R}\sim \basis{R_1},\life{R}\ni\tau_1\}\}.
\end{split}
\end{equation}
Let $U_1$ be a random variable uniformly distributed in $[0,1]$ independent of everything.
\item If $\eta(i,j)=0$ and $U_1 < \Q(\{i,j\}|\eta)$, then update $H \leftarrow H\setminus\{R_1\}$, $K \leftarrow K\cup\{R_1\}$. Go to (2).
\item If $\eta(i,j)=1$ or $U_1 > \Q(\{i,j\}|\eta)$, then update $H \leftarrow H\setminus\{R_1\}$. Go to (2).
\item Set $\g K^{V,0}=\g K$ and stop.
\end{enumerate}
\end{changemargin}

\textit{Algorithm to simulate a finite region of $\mu$}. We use Algorithm (2) to construct the set of kept rectangles $\g K^{V,0}$ of a finite region $V$. Define the graph with vertex set $V$ by
\begin{equation}
\eta(i,j) = \ind\left\lbrace \{i,j\} \in \{ \basis{R} : R\in \g K^{V,\,0}, \life{R}\ni 0\}\right\rbrace\,.
\end{equation} 
%

\section{Proofs}\label{sec:proofs}

In this section we present the main ideas for the proofs of the theorems of this paper. The results follow with the necessary modifications from the properties of the clan of ancestors studied  in \cite{ferrari2002perfect}.

\subsection{Proof of Theorem \ref{theo:clan_finite}}

\begin{proof}
[Proof of (1)] Without loss of generality set $V=\Zd{d}$. We want to partition the set of rectangles $\rec[0,t]\cup \reci{x}$ into a set of kept rectangles and a set of erased rectangles.
First, all initial rectangles $\reci{x}$ are kept. Since by assumption  $\A^{i,t}\cap \rec[0,t]$ is finite for any $i \in \Zd{d}$ and $t \geq 0$, we can partition this set in kept and erased rectangles following the same procedure as described in Section \ref{sec:finite_construction}. We denote the resulting kept set by $\kept^{i,\,t}_{\g x}[0,t]$ and the resulting erased set by $\mathbf{D}^{i,\,t}_{\g x}[0,t]$.
Denoting 
\[
\kept_{\g x}[0,t] : = \bigcup\limits_{i\in \Zd{d}}\kept^{i,\,t}_{\g x}[0,t] \qquad \qquad \mbox{ and } \qquad \qquad \mathbf{D}_{\g x}[0,t] : = \bigcup\limits_{i \in \Zd{d}}\mathbf{D}^{i,\,t}_{\g x}[0,t]\,,
\]
we have that 
\begin{equation}
\kept_{\g x}[0,t]\cup\mathbf{D}_{\g x}[0,t]= \rec[0,t]\cup \reci{x}\,.
\end{equation}

We define the process as in \eqref{eq:dep_const_kept} by 
\begin{equation}\label{eq:process_initial_graph}
\eta^{\,V,\,\g x}_{\,t}(i,j) = \ind \left\lbrace \{i,j\}\in \{\basis{R} : R\in\kept_{\g{x}}[0,t] ,\,  \life{R}\ni t  \}\,\right\rbrace \,.
\end{equation}

Notice that, for $V$ finite, this construction is equivalent to that presented of Section \ref{sec:finite_construction}.

For $V=\Zd{d}$, we want to show that $\eta^{\,V,\,\g x}_{\,t}$ has generator $\A^V$. To do that, denote $\eta_{\,t}=\eta^{\,V,\,\g x}_{\,t}$ and $\kept_{\g x}=\kept$. Define $\eta^{1}_{\,t,\, ij}$ as the graph obtained at time t
which coincides with $\eta_{\,t}$ for all edges other than $\{i,j\}$ and $ \eta^{1}_{\,t}(i,j)=1$. In the same way, $\eta^{0}_{\,t, \, ij}$ is the graph
which coincides with $\eta_{\,t}$ for all edges other than $\{i,j\}$ and $ \eta^{0}_{\,t}(i,j)=0$. By calculations very similar to Theorem 3.1 in \cite{ferrari2002perfect}, we have
\begin{equation}
\begin{split}
&\Scale[0.95]{\esp\left[ f(\eta_{\,t+h}) - f(\eta_{\,t})   \right]} \\ 
&\hspace{0.6cm}=\sum\limits_{\{i,j\}\in E_V}he^{-\beta L(i,j)-\beta M}\esp\left[\ind\left\lbrace  \eta_t(i,j)=0\right\rbrace \Q(\, \{i,j\}\, |\, \eta_{\,t}) \left[ f(\eta^{1}_{\,t, \, ij}) - f(\eta_{\,t})  \right]\right] \\
&\hspace{0.6cm}+\sum\limits_{\{i,j\}\in E_V}h\,\esp\left[ \eta_t(i,j)\left[ f(\eta^{0}_{\,t, \, ij}) - f(\eta_{\,t})  \right]\,\right]+o(h)\, ,
\end{split}
\end{equation}
which, dividing by $h$ and taking limit, gives, 
\begin{equation}
\dfrac{\mathrm{d}\esp f(\,\eta^{V,\,\g x}_{\,t}\,)}{dt} = \A^V\esp f(\,\eta^{V,\,\g x}_{\,t}\,)\,.
\end{equation}

The existence of an invariant measure follows by compactness, since our process is defined in the compact space $\X$. See Chapter 1 of \cite{liggett1985interacting}.\\

\proof[Proof of (2)] The assumption of no backward oriented percolation implies that the clan of ancestors $\A^{i,t}$ is finite for every $i\in \Zd{d}$ and $t\in \mathbb{R}$. This fact allow us to construct the set of kept rectangles $\kept$ as a partition of $\rec$ in the same way the set $\kept_{\g x}[0,t]$ was construct from $\rec[0,t]\cup\rec(\g x)$ in the proof of part (1) above. We just proceed clan by clan and simply ignore the initial rectangles of $\rec(\g x)$. Note that $\kept$ is both space and time invariant by construction. We define  $\eta_{\,t}$ as
\begin{equation}\label{eq:process_t}
\eta_{\,t}(i,j) = \ind \left\lbrace \{i,j\}\in \{\basis{R} : R\in\kept ,\,  \life{R}\ni t  \}\,\right\rbrace \,.
\end{equation}

By construction, the distribution of $\eta_{\,t}$ does not depend on $t$; hence its distribution is an invariant measure for the process.
\end{proof}

\subsection{Properties of the clan of ancestors}\label{sec:properties_clan}

Before starting the proofs of the main theorems of this paper we need to explore some properties of the clan of ancestors. To this end, let
\begin{equation}\label{eq:clan_of_set}
\A(I) = \bigcup_{i\in I}\A^{i,0}
\end{equation}
be the clan of ancestors of the set of vertices $I\in \Zd{d}$ at time $0$ constructed from $\rec$. In the same way we denote by $\A^V(I)$ the clan of ancestors of $I$ constructed from $\rec^V$, for $V$ finite.

The next results follows directly from the construction of the clan of ancestors. In this way, the proofs of Lemma \ref{lemma:space_ineq} and \ref{lemma:lemma_mixing} follow straightforward from the proof of Theorem 4.1-(3-4) in \cite{fernandez2001loss}.

\begin{lemma}\label{lemma:space_ineq}
Assume that there is no backward oriented percolation with probability 1. Let $V$ be a finite subset of $\Zd{d}$. For a measurable function $f$ with $\suppv{f}\subset V$ we have that
\begin{equation}\label{eq:ineq_space_suppv}
|\mu f - \mu_{_V} f | \leq 2\,||f||_{\infty}\,\P \left(\, \A(\suppv{f})\neq \A^{V}(\suppv{f})\, \right)\,.
\end{equation}
\end{lemma}

\begin{lemma}\label{lemma:lemma_mixing}
Assume that there is no backward oriented percolation with probability 1. Let $V$ be (finite or infinite) subset of $\Zd{d}$. For measurable functions $f$ and $g$ with $\suppv{f},\suppv{g}\subset V$, we have that 
\begin{equation}
|\mu fg - \mu f\mu\, g | \leq 2\,||f||_{\infty}||g||_{\infty}\,\P \left(\, \A(\suppv{f})\sim \widehat{\A}(\suppv{g})\, \right)\,,
\end{equation}
where $ \widehat{\A}(\suppv{g})$ has the same distribution as $ {\A}(\suppv{g})$ but is independent of $ \widehat{\A}(\suppv{f})$.
\end{lemma}

\subsection{Time length and space diameter of a clan of ancestors}\label{sec:time_space_clan}

In this section we focus on two particular properties of the clan of ancestors given by \eqref{def:clan_vertice}. First, define the \textit{time length} of the clan $\A^{i,t}$ as the length of the time interval between $t$ and the first birth in the family of ancestors of $(i,t)$ given by
\begin{equation}
TL(\A^{i,t}) = t - \sup\{ s : \life{R} \ni s,\,\mbox{ for some } R\in \A^{i,t}\}\,.
\end{equation}

Define the \textit{space diameter} of the clan $\A^{i,t}$ by
\begin{equation}
SD(\A^{i,t})= \max\{ d(\{i\},\basis{R}) : R \in \A^{i,t}\}\,.
\end{equation}

The space diameter of the clan $\A^{i,t}$ is the maximum distance between the vertex $i$ and the edges in the projection on $\Zd{d}$ of the clan $\A^{i,t}$.

\begin{proposition}\label{prop:space_clan} If $\beta > \beta^{\ast}$, then
\begin{enumerate}
\item the probability of backward oriented percolation is $0$.
\item for $a \leq \beta -\beta^{\ast}$ and  $M>0$,
\begin{equation}
\esp[\,e^{aSD(\A^{i,t})}\,]\, \leq \, \left(\dfrac{\alpha(\beta-a)e^{-aM}}{1-e^{-aM}\alpha(\beta-a)}\right).
\end{equation}
\item for $\til{\beta} \in (\beta^{\ast},\beta)$ and $M>0$,
\begin{equation}
\P(\,SD(\A^{i,t}) > k\,)\,\leq \, \left(\dfrac{e^{-(\beta-\til\beta)M}\alpha(\til{\beta})}{1-e^{-(\beta-\til\beta)M}\alpha(\til{\beta})}\right)e^{-(\beta-\til{\beta})k}.
\end{equation}
\item  for any positive $b$,
\begin{equation}
\P(\,TL(\A^{i,t})>bt)\, \leq \,\dfrac{1}{2}\alpha(\beta)e^{-(1-\alpha(\beta))bt}\,.
\end{equation}
\end{enumerate}
\end{proposition}

The proof of Proposition \ref{prop:space_clan}, is based on the construction of a branching process that dominates the backward percolation process as usual in this field and straightforward from the construction proposed in \cite{ferrari2002perfect}. This construction is based on a multitype branching process $\B_n$ defined on the space of rectangles in such a way that the ancestors play the role of the branches of the process. The process $\B_n$ induces a multitype random process in the set of edges, denoted by $\b_n$, such that, for a rectangle $R$ with basis $\{i,j\}$, $\b_n^{ij}(k,l)$ is the number of rectangles in the $n$th generation of ancestors of $R$ with basis $\{k,l\}$. The process $\b_n$ is a multitype branching process whose offspring distribution are Poisson with mean
\begin{equation}
\begin{split}
m(\{i,j\},\{k,l\}) &= \ind\{\,\{i,j\}\sim\{k,l\}\,\}e^{-\beta L(k,l)-\beta M} \int\limits_{0}^{\infty}e^{-t}dt\\
& =\ind\{\,\{i,j\}\sim\{k,l\}\,\}e^{-\beta L(k,l)-\beta M}\,.
\end{split}
\end{equation}

\begin{lemma}\label{lemma:mean_process}
Let $m^n$ be the $n$th power of the matrix $m$. We have that
\begin{equation}
\sum\limits_{\{k,l\}}m^n(\{i,j\},\{k,l\} ) \leq [\alpha(\beta)]^n
\end{equation}
where $\alpha$ is given by \eqref{def:alpha}.
\end{lemma}

\begin{proof}
We write
\begin{equation}\label{eq:lemma_mean1}
\begin{split}
&\text{\scalebox{0.85}{ $\sum\limits_{\{k,l\}}m^n(\{i,j\},\{k,l\}) $} }\\
&\hspace{0.8cm} \text{\scalebox{0.85}{$= \sum\limits_{\substack{\{k_1,l_1\} \\ \{k_1,l_1\} \sim \{i,j\} }}e^{-\beta L(k_1,l_1)-\beta M}\ldots \hspace{-15mm}\sum\limits_{\substack{\{k_{n-1},l_{n-1}\} \\ \{k_{n-1},l_{n-1}\} \sim \{k_{n-2},l_{n-2}\} }} \hspace{-10mm} e^{-\beta L(k_{n-1},l_{n-1})-\beta M}
 \hspace{-5mm} \sum\limits_{\substack{\{k,l\} \\ \{k,l\} \sim \{k_{n-1},l_{n-1}\} }}  \hspace{-5mm} e^{-\beta L(k,l)-\beta M}$} }\\
&\hspace{0.8cm} \Scale[0.85]{ \leq e^{-\beta Mn}\left( \sup\limits_{\{i',j'\}}\sum\limits_{\substack{\{k,l\} \\ \{k,l\} \sim \{i',j'\} }}e^{-\beta L(k,l)} \right)^n} \\
&\hspace{0.8cm} = e^{-\beta Mn}\left( \sup\limits_{\{i',j'\}}\sum\limits_{\substack{k\in\Zd{d}\\k\neq i',j'}} \left( e^{-\beta L(i',k)} + e^{-\beta L(j',k)} \right) \right)^n\,.
\end{split}
\end{equation}

Define $\overline{B}^{\,i}_s=\{j\in \Zd{d}, j\neq i: L(i,j)=s\}$. Using the fact that 
\[|\overline{B}^{\,i}_s|\leq {s+d-1\choose d-1}2^d\, \]
we have that
\begin{equation}\label{eq:lemma_mean2}
\begin{split}
\sum\limits_{\substack{k\in\Zd{d}\\k\neq i',j'}} \left( e^{-\beta L(i',k)} + e^{-\beta L(j',k)} \right) &=2\left(\sum\limits_{s=1}^{\infty} \sum\limits_{k\in\overline{B}^{\,i}_s} e^{-\beta s} \right)\\
& \leq 2^{d+1}\sum\limits_{s=1}^{\infty}e^{-\beta s}{s+d-1\choose d-1} \\
&= 2^{d+1}\left( \frac{1}{(1-e^{-\beta})^d}-1\right)\,.
\end{split}
\end{equation}


By \eqref{eq:lemma_mean1} and \eqref{eq:lemma_mean2} the result follows.
\end{proof}

\begin{proof}[Proof of Proposition \ref{prop:space_clan}] 
(1) By the construction of the branching process, for a rectangle $R$ with basis $\{i,j\}$, we have that the number of rectangles in the $n$th generation of ancestors of $R$ is dominated by $\sum\limits_{\{k,l\}}\b_n^{ij}(k,l)$. 

To prove that there is no backward oriented percolation, it is enough to prove that, for fixed $\{i,j\}$,
\begin{equation}\label{eq:proof_space1}
\P\left(\sum\limits_{\{k,l\}}\b_n^{ij}(k,l)\neq 0 \mbox{ for infinitely many n}\right)=0\,.
\end{equation}

Since $\beta > \beta^{\ast}$, we have by Lemma \ref{lemma:mean_process} that
\begin{equation}
\sum\limits_{n}\sum\limits_{\{k,l\}}m^n(\{i,j\},\{k,l\} )< \infty
\end{equation}

By Borel-Cantelli lemma \eqref{eq:proof_space1} follows.\\

\proof[Proof of (2)] We write 
\begin{equation}\label{proof:3ineq_1}
\esp[\,e^{aSD(\A^{i,t})}\,] = \sum\limits_{l=1}^{\infty} e^{al}\,\P(\,SD(\A^{i,t})=l\,)\,.
\end{equation}

Using the branching process $\b_n$, we have
\begin{equation}\label{proof:3ineq_2}
\begin{split}
&\P(\,SD(\A^{i,t})=l\,) \leq \sum\limits_{k=1}^{\infty}\sum\limits_{\substack{ \{r_1,s_1\}\\\cdots \\ \{r_k,s_k\} }}\ind\{\max\{ d(\{i\},\{\{r_1,s_1\},...,\{r_k,s_k\}\}\}=l\}\\
& \hspace{2.8cm}\times\P( \{r_1,s_1\}\ni i, \b_n^{r_1s_1}(r_2,s_2)\geq 1, \cdots, \b_n^{r_{k-1}s_{k-1}}(r_k,s_k)\geq 1)
\end{split}
\end{equation}

By the Markovian property of $\b_n$, we have
\begin{equation}\label{proof:3ineq_3}
\begin{split}
& \P( \{r_1,s_1\}\ni i, \b_n^{r_1s_1}(r_2,s_2)\geq 1, \cdots, \b_n^{r_{k-1}s_{k-1}}(r_l,s_l)\geq 1)\\
& \qquad = \P( \{r_1,s_1\}\ni i)\P( \b_n^{r_1s_1}(r_2,s_2)\geq 1) \cdots\P( \b_n^{r_{k-1}s_{k-1}}(r_k,s_k)\geq 1)\\
&\qquad \leq \ind \text{\scalebox{0.7}{$\{\{r_1,s_1\}\ni i, \{r_1,s_1\}\sim\{r_2,s_2\}, \cdots,\{r_{k-1},s_{k-1}\}\sim\{r_k,s_k\}\}$}}\prod\limits_{j=1}^ke^{-\beta L(r_j,s_j)-\beta M}\,.
 \end{split}
\end{equation}

We can also write
\begin{equation}\label{proof:3ineq_4}
\ind\text{\scalebox{0.85}{$\{\max\{ d(\{i\},\{\{r_1,s_1\},...,\{r_k,s_k\}\}\}=l\}$}}\,e^{al} \leq  e^{a\sum\limits_{j=1}^kL(r_j,s_j)} \,.
\end{equation}

Using \eqref{proof:3ineq_1}-\eqref{proof:3ineq_4} we have that
\begin{equation}
\begin{split}
&\Scale[0.93]{\esp[\,e^{aSD(\A^{i,t})}\,] }\\
&\quad \leq \Scale[0.94]{\sum\limits_{k=1}^{\infty}\sum\limits_{\substack{ \{r_1,s_1\}\\\dots \\ \{r_k,s_k\} }}}\ind \Scale[0.7]{\{\{r_1,s_1\}\ni i, \{r_1,s_1\}\sim\{r_2,s_2\}, \dots,\{r_{k-1},s_{k-1}\}\sim\{r_k,s_k\}\}} \Scale[0.93]{e^{-\sum\limits_{m=1}^k(\beta-a) L(r_m,s_m)-\beta M}}\\
&\quad=\Scale[0.94]{\sum\limits_{k=1}^{\infty}\sum\limits_{\substack{j\in \Zd{d}\\ j\neq i}} }\Scale[0.86]{e^{-(\beta-a)L(i,j)-\beta M}}\Scale[0.94]{\sum\limits_{\substack{ \{r_2,s_2\}\\\cdots \\ \{r_{k},s_{k}\} }}}\ind\text{\scalebox{0.6}{$\{\{r_2,s_2\}\sim\{i,j\}, \dots,\{r_{k-1},s_{k-1}\}\sim\{r_{k},s_{k}\}\}$}}\Scale[0.86]{e^{-\sum\limits_{m=2}^k(\beta-a) L(r_m,s_m)-\beta M}}\\
&\quad =\Scale[0.9]{\sum\limits_{k=1}^{\infty}\sum\limits_{\substack{j\in \Zd{d}\\j\neq i}}}\Scale[0.8]{ e^{-(\beta-a)L(i,j)-\beta M}}\Scale[0.8]{\sum\limits_{\substack{\{r_2,s_2\} \\ \{r_2,s_2\} \sim \{i,j\} }}}\Scale[0.84]{e^{-(\beta-a) L(r_2,s_2)-\beta M}\dots} \Scale[0.85]{\sum\limits_{\substack{\{r_{k},s_{k}\} \\ \{r_{k},s_{k}\} \sim \{r_{k-1},s_{k-1}\} }}} \Scale[0.8]{e^{-(\beta-a) L(r_{k},s_{k})-\beta M}}\\
&\quad \leq\, \sum\limits_{k=1}^{\infty}(e^{-aM}\alpha(\beta-a))^k\,.
\end{split}
\end{equation}
The result follows for $\beta -a \geq \beta^{\ast}$.

\proof[Proof of (3)] Setting $a=(\beta-\til\beta)$, for $\til\beta\in (\beta^{\ast},\beta)$, the result follows from Chebyshev inequality.

\proof[Proof of (4)] It is possible to construct a continuous time branching process $\B^{i,t}$ looking backward on time (of the original marked Poisson process) and associating the ancestors as the branches of the branching process that dominates the clan of ancestors $\A^{i,t}$. In this construction, births in the original marked Poisson process corresponds to the disappearance of branches in the branching process. For a rectangle $R$ with $\basis{R}=\{i,j\}$ and $\birth{R}=0$, we define the continuous time branching process $\psi_t^{\,ij}(k,l)$ by the number of edges of type $(k,l)$ present at time $t$ (of this process) whose initial graph is $\g x^{ij}$, where $x^{ij}(r,s)=\ind\{\{r,s\}=\{i,j\}\}$, that is, $\g x^{ij}$ is a graph with only the edge $\{i,j\}$. It is immediate from construction to see that in this process, each edge $\{i,j\}$ lives a mean-one exponential time after which it dies and gives birth to $C_{kl}$ edges $\{k,l\}$, $\{k,l\}\in E$, with probability
\begin{equation}
\prod_{\{k,l\}}\dfrac{e^{m(\{i,j\},\{k,l\})}m(\{i,j\},\{k,l\})^{C_{kl}}}{C_{kl}!}.
\end{equation}

The infinitesimal generator of the process is given by
\begin{equation}\label{eq:generator_continuous_branching}
\Scale[0.8]{Lf(\psi) =} \sum\limits_{\Scale[0.6]{\{i,j\}\in E}}\Scale[0.78]{\psi(\{i,j\})}\sum\limits_{\Scale[0.65]{\eta\in\mathcal{Y}_0^{ij}}}\prod\limits_{\Scale[0.75]{\substack{\{k,l\}: \\ \eta(\{k,l\})\geq 1}}} \Scale[0.78]{\dfrac{e^{m(\{i,j\},\{k,l\})}m(\{i,j\},\{k,l\})^{\eta(\{k,l\})}}{\eta(\{k,l\})}[f(\psi+\eta-\g x^{ij})-f(\psi)]}
\end{equation}
where $\psi, \eta \in \mathcal{Y}_0=\{ \psi \in \mathbb{N}^E;\, \sum_{\{k,l\}}\psi(\{k,l\})<\infty\}$,  $f: \mathcal{Y}_0\to\mathbb{N}$ and $\mathcal{Y}_0^{ij}=\{ \psi \in \mathcal{Y}_0;\, \psi(\{k,l\})\geq 1 \mbox{ implies } \{k,l\}\sim\{i,j\} \}$.

We define the mean number of edges of type $\{k,l\}$ in $\psi_t$ and its sum over $\{k,l\}$ by
\begin{equation}
M_t(\{i,j\},\{k,l\})=\esp[\,\psi_t^{\,ij}(k,l)\,], \qquad R_t(\{i,j\})=\sum_{\{k,l\}}M_t(\{i,j\},\{k,l\})\,.
\end{equation}

By  Lemma 5.2 in \citep{fernandez2001loss}, we have that
\begin{equation}\label{eq:continuous_branching_ineq}
\P\left(\,\sum_{\{k,l\}}\psi_t^{\,ij}(k,l) > 0\,\right) \leq  R_t(\{i,j\})\leq e^{(\alpha(\beta)-1)t}.
\end{equation}

Therefore,  have that
\[TL(\A^{i,t})\leq TL(\B^{i,t}), \mbox{ for } i \in \Zd{d}\,,\]
and
\begin{equation}
\sum_{\{k,l\}}\psi_t^{\,ij}(k,l)=0 \,\mbox{ implies } \, TL(\A^{i,\,0}) \leq t\,.
\end{equation}

Thus, by \eqref{eq:continuous_branching_ineq}
\begin{equation}
\begin{split}
\P(TL(\A^{i,\,0}) > t) &\leq \sum_{j\in \Zd{2}}\P(\eta_0({i,j}))R_t(\{i,j\})\\
& \leq e^{-(1-\alpha(\beta))t}\sum_{j\in \Zd{2}}e^{-\beta L(i,j)-\beta M}\\
& \leq \dfrac{1}{2}\alpha(\beta)e^{-(1-\alpha(\beta))t}
\end{split}
\end{equation}
\end{proof}

\subsection{Proofs of Existence and Uniqueness}

\begin{proof}[Proof of Theorem \ref{theo:existence}]
It is sufficient to prove that $\alpha(\beta) < \infty$ implies that, with probability $1$, 
\begin{equation}\label{proof:theo_existence_1}
\A^{i,t}\cap \rec[0,t] \mbox{ is finite, for any } i\in \Zd{d} \mbox{ and } t\geq 0\,.
\end{equation}

Thus, Theorem \ref{theo:existence} follows from Theorem \ref{theo:clan_finite}-(1). 

To prove \eqref{proof:theo_existence_1} it is enough to prove that $\A^{i,0}\cap \rec[-t,0]$ is finite with probability $1$, by time translation invariance.
Similary, in the construction of the continuous time branching process in the proof Proposition \ref{prop:space_clan}-(4), for $R$ with $\basis{R}=\{i,j\}$ and $\birth{R}=0$, we define
the process $\widetilde\psi_t^{\,ij}$ which indicates all edges born in $[0,t]$ in the process $\psi_t^{\,ij}$. Notice that, for all $j\in \Zd{d}$
\begin{equation}\label{proof:theo_existence2}
|\A^{i,0}\cap \rec[-t,0]\,| \leq |\,\{R\in \B^{i,0}: \birth{R}\in [-t,0]\}\,| \leq \sum_{\{k,l\}}\widetilde\psi_t^{\,ij}(k,l)\,.
\end{equation}

Taking the expectation of the right-hand side of \eqref{proof:theo_existence2} and proceeding in the same way as done for $\psi_t^{\,ij}$ we get
\begin{equation}
\begin{split}
\esp\left( \sum_{\{k,l\}}\widetilde\psi_t^{\,ij}(k,l) \right) &= \sum_{\{k,l\}}[e^{tm}](\{k,l\}) \\
& = \sum_{\{k,l\}}\sum_{n\geq 1}\dfrac{t^n m^n(\{i,j\},\{k,l\})}{n!}\\
& \leq e^{t\alpha(\beta)} < \infty\,
\end{split}
\end{equation}
since $\alpha(\beta) < \infty$.
Thus, we conclude that $|\A^{i,0}\cap \rec[-t,0]\,|$ is finite with probability $1$.
\end{proof}

\begin{proof}[Proof of Theorem \ref{theo:uniqueness}]
(1) For $\beta> \beta^{\ast}$, we have by Proposition \ref{prop:space_clan} that there is no backward oriented percolation with probability $1$. The uniqueness of $\mu$ is guaranteed by Theorem \ref{theo:clan_finite}-(2) and the construction of the perfect simulation algorithm to simulate from $\mu$. \\

(2) Since for $\beta > \beta^{\ast}$ there is no backward oriented percolation, with probability 1, we have, for a measurable function $f$ with $\suppv{f}\subset V$ that $\A(\suppv{f})$ and $\A^V(\suppv{f})$ are finite. As $V \rightarrow \Zd{d}$, we have that 
\begin{equation}
\P \left(\, \A(\suppv{f})\neq \A^{V}(\suppv{f})\, \right) \to 0\,.
\end{equation}

Using Lemma \ref{lemma:space_ineq}, we have the weak convergence of $\mu_{_V}$ to $\mu$. To prove that $\mu$ is concentrated on the set of graphs with finite degree we use \eqref{eq:exp_process_inf} and \eqref{eq:ineq_free_dep_exp} to get
\begin{equation}
\mu\, \dg{i}{} = \esp\, \dg{i}{}(\eta_{\,t}) = \sum\limits_{j\in \Zd{d}}\esp\, \eta_{\,t}(i,j) \leq  \sum\limits_{j\in \Zd{d}}e^{-\beta L(i,j)-\beta M} \leq \frac{1}{2}\alpha(\beta)\,.
\end{equation}
\end{proof}

\begin{proof}[Proof of Theorem \ref{theo_space_bounds}]
For a measurable function $f$ with $\suppv{f}\subset V$ we can write
\begin{equation}\label{proof:theo_space_eq1}
\begin{split}
\P \left(\, \Scale[0.95]{\A(\suppv{f})\neq \A^{V}(\suppv{f})}\, \right) &\leq \P \left(\, \Scale[0.95]{SD\left(\A(\suppv{f})\right)\geq d(\suppv{f},V^c)}\, \right)\\
& \leq \sum\limits_{i \in \suppv{f}}\P \left(\, \Scale[0.95]{SD\left(\A^{i,\,0}\right)\geq d(\{i\},V^c)}\, \right)\,.
\end{split}
\end{equation}

Using Proposition \ref{prop:space_clan}-(3) we have that the RHS of \eqref{proof:theo_space_eq1} is upper bounded by
\begin{equation}
\left(\dfrac{e^{-(\beta-\til\beta)M}\alpha(\til{\beta})}{1-e^{-(\beta-\til\beta)M}\alpha(\til{\beta})}\right)\sum\limits_{i \in \suppv{f}}e^{ -(\beta-\tilde\beta)d(\{i\},V^c)}\,.
\end{equation}

Thus, the result follows by Lemma \ref{lemma:space_ineq}.
\end{proof}

\begin{proof}[Proof of Theorem \ref{theo:mixing}]
For measurable functions $f$ and $g$ with $\suppv{f}\subset V,\, \suppv{g}\subset V$ we write
\begin{equation}\label{eq:proof_mixing1}
\begin{split}
\Scale[0.91]{\P \left(\, \A(\suppv{f})\sim \widehat{\A}(\suppv{g})\, \right)} &\leq \sum\limits_{ \substack{ i \in \suppv{f} \\  j \in \suppv{g}}}\P(\A^{i,0}\sim \widehat\A^{j,0})\\
&\leq \sum\limits_{ \substack{ i \in \suppv{f} \\  j \in \suppv{g}}}\P(SD(\A^{i,0})+SD(\widehat\A^{j,0}) \geq L(i,j)),
\end{split}
\end{equation}
where $ \widehat{\A}(\suppv{g})$ has the same distribution as $ {\A}(\suppv{g})$ but it is independent of $ \widehat{\A}(\suppv{f})$.
By \eqref{eq:exp_prob_finite} and \eqref{eq:exp_process_inf} we write
\begin{equation}
\mu_{_V}(fg) - \mu_{_V}f\mu_{_V}g = \esp[\,f(\eta_0)g(\eta_0)\,] - \esp[\,f(\eta_0)\,]\esp[\,g(\eta_0)\,]\,.
\end{equation}

Using the following inequality, which is valid for independent random variables $S_1$ and $S_2$,
\[\P(S_1+S_2 \geq l)\leq \sum\limits_{k=1}^l\P(S_1 \geq k)\P(S_2\geq l-k)\]
and Proposition \ref{prop:space_clan}-(3) in the right-hand side of \eqref{eq:proof_mixing1} we get
\begin{equation}\label{eq:proof_mixing2}
\begin{split}
&\P \left(\, \A(\suppv{f})\sim \widehat{\A}(\suppv{g})\, \right)\\
& \qquad \leq\, \left(\dfrac{e^{-(\beta-\til\beta)M}\alpha(\til{\beta})}{1-e^{-(\beta-\til\beta)M}\alpha(\til{\beta})}\right)^2\sum\limits_{ \substack{ i \in \suppv{f} \\  j \in \suppv{g}}}L(i,j)e^{-(\beta-\til\beta)L(i,j)}\,.
\end{split}
\end{equation}

Combining \eqref{eq:proof_mixing2} and  Lemma \ref{lemma:lemma_mixing} the result follows.
\end{proof}

\begin{proof}[Proof of Theorem \ref{theo:central_limit}]
The main idea of this proof is to use the central limit theorem for stationary mixing random fields proved in  \cite{bolthausen1982central} combined with the exponential mixing property given in Theorem \ref{theo:mixing}. To this end, we write $X_i=\tau_if$ and let $\mathcal{A}_V$ be the $\sigma-$algebra generated by $\{X_i: i\in V\}$, for $V\subset \Zd{d}$ finite. Define,
\begin{equation}
\begin{split}
\alpha_{k,l}(n)=&\sup\{\mid \P(A_1\cap A_2) - P(A_1)P(A_2)\mid \,:\, A_1 \in \mathcal{A}_{V_1}, A_2 \in \mathcal{A}_{V_2},\\
&\hspace{1cm} |V_1|\leq k, |V_2|\leq l, d(V_1,V_2) \geq n \}
\end{split}
\end{equation}

We use here the result stated in Remark 1 in \cite{bolthausen1982central}, that says that if there exists a $\delta>0$ such that $||X_i||_{2+\delta}<\infty$ and 
\begin{equation}\label{eq:proof_central_1}
\sum\limits_{n=1}^{\infty}n^{d-1}(\alpha_{2,\infty}(n))^{\frac{\delta}{2+\delta}} < \infty
\end{equation}
then \eqref{eq:theo_central_1} and \eqref{eq:theo_central_2} hold.
In order to prove \eqref{eq:proof_central_1} we write
\begin{equation}
\alpha_{2,\infty}(n) = \sup\limits_{a,g_1,g_2}\mid \mu(g_1g_2) - \mu g_1\mu g_2\mid\,
\end{equation}
where the supremum is taken over the set of $a\in \Zd{d}$, $g_1$ in the set of indicator functions with vertex support on $\suppv{f}\cup\tau_a\suppv{f}$ and $g_2$ in the set of indicator functions with vertex support in 
\[\bigcup\{ \tau_j\suppv{f}\,:\, j\in \Zd{d} \text{ and } L(i,j)\geq n, \forall i \in \suppv{f}\}\,.\]

Using Theorem \ref{theo:mixing}, using the fact that $||g_1||_{\infty}=||g_2||_{\infty}=1$ and taking $C=2\left(\frac{e^{-(\beta-\til\beta)M}\alpha(\til{\beta})}{1-e^{-(\beta-\til\beta)M}\alpha(\til{\beta})}\right)^2$ we have that
\begin{equation}\label{eq:proof_central_2}
\begin{split}
\alpha_{2,\infty}(n) &\leq C\sum\limits_{ \substack{ i \in \suppv{g_1} \\  j \in \suppv{g_2}}}L(i,j)e^{-(\beta-\til\beta)L(i,j)}\\
& \leq 2C\suppv{f}\sum\limits_{  s=n}^{\infty}s{s+d-1\choose d-1}e^{-(\beta-\til\beta)s}\\
&\leq 2C\suppv{f}e^{-(\beta-\til\beta)n}\sum\limits_{  t=0}^{\infty}(t+n){t+n+d-1\choose d-1} e^{-(\beta-\til\beta)t}\\
&\leq 2C\suppv{f}e^{-(\beta-\til\beta)n}\sum\limits_{  t=0}^{\infty}(t+n){\left(\frac{e(t+n+d-1)}{d-1}\right)}^{d-1} e^{-(\beta-\til\beta)t}\,.
\end{split}
\end{equation}
where the second inequality follows because $|\suppv{g_1}| \leq 2|\suppv{f}|$ and the last inequality follows from ${n\choose x}\leq \left(\frac{en}{x}\right)^x$. Because the right hand side of \eqref{eq:proof_central_2} is of order $n^de^{-(\beta-\til\beta)n}$, the condition given by \eqref{eq:proof_central_1} is satisfied. 
\end{proof}

\section{Final Considerations} \label{sec:conclusion}
There are several probabilistic models on infinite graphs proposed in the literature that incorporates dependencies among the edges, but do not take into account distance between the edges. On the other hand, other models incorporate distance between the edges but do not take into account dependencies. The purpose of this work is to propose a Spatial  Gibbs Random Graph Model that incorporates both, the statistics of the graph and the metric space where the vertices are located. The first question to be considered is obviously the existence and uniqueness of such model in infinite volume. This question is still an open problem for models describing more complex relations in the network. In this work we show that in order to have a positive answer to this question we need Assumption \ref{assumption1}. Roughly speaking, this assumption states that when one edge is modified in the graph, the statistics of the graph will be only affect by the edges connecting the modified edge endpoints.  This assumption includes the statistics that are commonly used in the literature to describe transitivity relations in graphs, such as, triangles, k-stars and degrees. Using the graphical construction based on the clan of ancestors not only prove existence and uniqueness of the model but also obtain theoretical results such as law of large numbers, central limit theorems and exponential mixing. These results can be explored to general metric spaces . For example, it is possible to consider that the nodes are obtained by a point process on $\mathbb{R}^d$, such as an homogeneous Poisson process. In this setting, \citep{ferrari2010gibbs} proved the existence of the infinite volume measure for the specific model described in Example \ref{ex:ferrari}. Theoretical results for more general models are still under investigation.  \\



\textit{Acknowledgements.} 
This work
was produced as part of the activities of FAPESP Research, Innovation and Dissemination Center for Neuromathematics, grant 2013/07699-0. AC was supported by a FAPESP's scholarship 2017/25925-4. Part of this work was completed while AC was a Visiting Postdoctoral Researcher at University of Michigan. She thanks the support and hospitality of this institution. NLG was partially supported by grants CNPq 302598/2014-6 and FAPESP  2017/10555-0.

\bibliographystyle{elsarticle-harv} 

\bibliography{references}


%
%
%
\end{document}